\documentclass[11pt,a4paper]{article}
\usepackage[T1]{fontenc}
\usepackage[utf8]{inputenc}
\usepackage{lmodern}
\usepackage[left=27mm,right=27mm,top=25mm,bottom=27mm]{geometry}
\usepackage{amsmath,amssymb,amsthm,mathtools}
\usepackage{microtype}
\usepackage{enumitem}
\usepackage[authoryear,round]{natbib} 
\usepackage{xcolor}
\usepackage{hyperref}
\hypersetup{
  colorlinks=true,
  linkcolor=blue!45!black,
  citecolor=blue!45!black,
  urlcolor=blue!45!black,
  pdftitle={Couplings with prescribed maxima and a short proof of the characterization of quasi-copulas},
  pdfauthor={},
  pdfsubject={A constructive maximum coupling and the classical track characterization}
}
\newtheorem{theorem}{Theorem}
\newtheorem{lemma}[theorem]{Lemma}
\newtheorem{corollary}[theorem]{Corollary}

\theoremstyle{remark}
\newtheorem*{remark}{Remark}
\newcommand{\I}{[0,1]}
\newcommand{\PP}{\mathbb{P}}

\newcommand{\R}{\mathbb{R}}
\newcommand{\Pp}{\mathbb{P}}

\newcommand{\doi}[1]{\href{https://doi.org/#1}{doi:\,\nolinkurl{#1}}}

\definecolor{JAcolor}{RGB}{190,30,45}

\title{Couplings with prescribed maxima: A short proof of the quasi-copula characterization}

\author{
Jonathan Ansari\\[0.4em]
\small Department of Mathematics\\
\small Paris Lodron University Salzburg, Austria
}

\date{October 21, 2026}

\begin{document}
\maketitle
\vspace{-2.2em}

\begin{abstract}
We provide a short probabilistic proof of the classical characterization of multivariate quasi-copulas, addressing the call for a simpler and more transparent proof expressed by \cite{Sempi2017}. Our approach is based on a coupling construction for random variables with fixed marginal distributions and a prescribed distribution of their maximum. In our main result, we show that such a coupling exists if and only if the maximum is stochastically no smaller than each component and the distribution of the maximum is dominated by the sum of the marginal measures. The construction applies to arbitrary distributions and relies only on a decomposition of the maximum law and monotone quantile couplings. A direct consequence of our main result is the existence of copulas with prescribed values along an increasing track.
\end{abstract}


\section{Introduction}
\label{sec:intro}

Quasi-copulas were introduced by \citet{ANS1993}
in the bivariate case and extended to arbitrary
dimensions by \citet{NQSS1996}. They generalize
copulas by retaining their boundary conditions,
coordinatewise monotonicity, and Lipschitz
continuity, without requiring $d$-increasingness.
Quasi-copulas arise naturally as pointwise infima
and suprema of families of copulas, which need not
be copulas themselves; see \citet{NQRLU2004}.
Consequently, they play an important role in
dependence modeling, for instance, in deriving bounds for certain functional under
partial dependence information. Examples include
bounds for copulas with prescribed measures of
association, such as Kendall's tau or Spearman's
rho \citep{NQRLU2001}, as well as model-independent
price bounds for multi-asset derivatives
\citep{LP2017,ALNS2024}.
Applications to fuzzy aggregation are discussed
in \citet{AMB2020}.


Originally, quasi-copulas were defined as functions that agree along any increasing
track connecting $\mathbf{0}$ and $\mathbf{1}$
in $[0,1]^d$ with a copula, which may depend on the track.
Theorem \ref{thm:quasi} below is the classical result of \citet{GQRS1999} for
$d=2$ and \citet{CT2001} for arbitrary $d$. It shows that the original definition is equivalent to
a simple analytic characterization in terms
of uniform marginals, coordinatewise monotonicity,
and an $\ell^1$-Lipschitz condition.
This characterization is now commonly used
as the definition of quasi-copulas.

Although the analytic conditions are elementary,
the converse implication is considerably more involved:
one must show that, for every increasing track,
a function satisfying these conditions agrees with
some copula along this track.
 The multivariate proof of
\citet{CT2001} relies on finite interpolation
problems, constructions involving partial
measures, and a compactness argument.
In his survey, \citet[Section~3]{Sempi2017}
explicitly expressed the hope that this proof
would eventually be replaced by a simpler
and more transparent one.

In this paper, we provide a short probabilistic proof of this
classical characterization. Our approach is based on Theorem \ref{thm:maximum}---the main result of this paper---which characterizes the existence of coupling for random variables with given marginal distributions
and a prescribed distribution of their maximum.
Distributions of order statistics under
marginal information have a substantial
history; see, e.g., \citet{Rychlik1994,Rychlik1996}.
Related results include constructions of
copulas with prescribed diagonal sections
\citep{Jaworski2009,DFT2014} and copulas
of order statistics \citep{LD2014}.

A closely related problem is the
optimal transport problem studied by
\citet{BKLPE2022}, where a cost functional is maximized over the class of couplings with prescribed marginal and maximum distributions.
Their work provides a transport-duality approach,
whereas our Corollary \ref{cor_max} yields a simple
feasibility criterion and the existence of
couplings without solving an optimization problem.


The remainder of the paper is organized
as follows. Section~\ref{sec2} presents the main
coupling result. Section~\ref{sec3} discusses its
connection to the maximum transport problem.
Section~\ref{sec4} derives the existence of copulas
with prescribed values along a track.
Section~\ref{sec:quasi} establishes the classical
characterization of quasi-copulas by a
short probabilistic argument. Finally,
Section~\ref{sec:proofs} provides the constructive proof
of Theorem~\ref{thm:maximum}.

\section{Main result}\label{sec2}

The following result is the probabilistic statement underlying this paper.
For a distribution function $F$ on $\R$, write $\mu_F$ for its probability
measure. For finite Borel measures, $\alpha\le\beta$ means
$\alpha(B)\le\beta(B)$ for every Borel set $B$.

\begin{theorem}[Couplings with prescribed maximum]
\label{thm:maximum}
Let $d\ge2$, and let $F_1,\ldots,F_d,G$ be distribution
functions on $\R$. The following are equivalent:
\begin{enumerate}[label = (\roman*)]
    \item There exists a random vector $(X_1,\ldots,X_d)$ such that \(X_i\sim F_i\) for all \(i\in \{1,\ldots,d\}\) and \(\max_{1\le i\le d}X_i\sim G\).
    \item \label{eq:feasible} \(G(t)\le F_i(t)\) for all \(t\in\R\) and \(i\in \{1,\ldots,d\}\) and \(\mu_G\le\sum_{i=1}^d\mu_{F_i}\).
\end{enumerate}
\end{theorem}

The first condition in \ref{eq:feasible} states that the
prescribed maximum is stochastically no smaller than any of
its components.
The second condition follows from the fact that the maximum
of finitely many random variables must be attained by at
least one of them; hence, its probability on any Borel set cannot
exceed the sum of the corresponding marginal probabilities.
What is perhaps surprising is that these two elementary
necessary conditions are already sufficient. In particular,
no further compatibility conditions are required, regardless
of the unspecified dependence structure.
The proof in Section~\ref{sec:proofs} provides an explicit
coupling construction, without continuity assumptions or
approximation arguments.

\section{Existence of couplings with prescribed maximum law}\label{sec3}

Denote by \(\mathcal{P}(\R^d)\) the set of probability measures on (the Borel \(\sigma\)-algebra of) \(\R^d\).
For probability measures $\mu_1,\ldots,\mu_d,\gamma\in \mathcal{P}(\R)$, define
\[
 \Pi_{\max}(\mu_1,\ldots,\mu_d;\gamma):
 =\{\pi\in \mathcal{P}(\R^d):\operatorname{proj}_i{}_\# \pi=\mu_i ~ \text{for all } i,\ (\max)_\#\pi=\gamma\},
\]
where \(\R^d\ni x=(x_1,\ldots,x_d) \mapsto\operatorname{proj}_i(x):= x_i\) denotes the projection onto the \(i\)th component, and \(f_\#\pi(B):=\pi(f^{-1}(B))
=\pi(\{x\in\mathbb R^d:f(x)\in B\})\), \(B\subset \R\) Borel, denotes the pushforward measure of \(\pi\) under a measurable function \(f\colon \R^d\to \R\). Hence, \(\Pi_{\max}\) contains all \emph{couplings} of \(\mu_1,\ldots,\mu_d\) with prescribed maximum distribution.

\citet[Corollary~2.5]{BKLPE2022} consider an optimal transport problem where they optimize a cost functional over the set \(\Pi_{\max}\). The following result is a consequence of Theorem \ref{thm:maximum} and provides a simple characterization of non-emptiness of $\Pi_{\max}$. 

\begin{corollary}[Couplings with given maximum law]\label{cor_max}
Denote by \(F_i\) and \(G\) the distribution functions of \(\mu_i\) and \(\gamma\), respectively. Then the following statements are equivalent:
\begin{enumerate}[label = (\roman*)]
    \item \(\Pi_{\max}(\mu_1,\ldots,\mu_d;\gamma)\ne\varnothing\).
    \item \(G(t)\le F_i(t)\) for all \(i,t\), and \(\gamma\le\sum_{i=1}^d\mu_i\).
\end{enumerate}
\end{corollary}

\begin{proof}
There exists a probability measure $\pi$ on $\R^d$ with marginals $\mu_1,\ldots,\mu_d$ such that
$(\max)_\#\pi=\gamma$ if and only if there exists a random vector $(X_1,\ldots,X_d)$ satisfying \(X_i\sim F_i\), \(i=1,\ldots,d\), and \(\max_{1\le i\le d}X_i\sim G\). Hence, the assertion follows from Theorem~\ref{thm:maximum}.
\end{proof}

From couplings of general distributions on \(\R\), we now turn to copulas which couple uniform distributions on \([0,1]\).

\section{Existence of copulas with prescribed values along a track}\label{sec4}

A (\(d\)-dimensional) copula is a distribution function \(C\colon [0,1]^d\to [0,1]\) with \emph{uniform marginals} in the sense that \(C(u) = u_i\) for all \(i\in \{1,\ldots,d\}\) and for all \(u = (u_1,\ldots,u_d)\in [0,1]^d\) with \(u_j=1\) for all \(j\ne i\). Copulas describe the dependence structure of multivariate
distributions through Sklar's theorem: For
every $d$-variate distribution function
$F\colon\R^d\to[0,1]$ with marginal distribution
functions $F_1,\ldots,F_d$ there is a copula \(C\) such that
\begin{align}
F(x)=C(F_1(x_1),\ldots,F_d(x_d))
\qquad\text{for all }x=(x_1,\ldots,x_d)\in\R^d.
\label{eq_Sklar}
\end{align}
If the marginal distribution functions are continuous,
the copula $C$ is uniquely determined. Conversely,
any copula $C$ and any univariate distribution functions
$F_1,\ldots,F_d$ define a $d$-variate distribution
function through \eqref{eq_Sklar}.
Thus, Sklar's theorem allows the dependence structure
to be modeled separately from the marginal distributions.
We refer to \cite{DS2016} for a comprehensive
introduction to copula theory.

Quasi-copulas, defined via Theorem~\ref{thm:quasi} below, generalize copulas. They were originally defined as functions that coincide on every track with a copula.
To this end, recall that a \emph{track} is (the image of) a continuous, coordinatewise nondecreasing map $t\mapsto(F_1(t),\ldots,F_d(t))$ from $\I$ to $\I^d$, with each
$F_i(0)=0$ and $F_i(1)=1$. 
Further, note that the measure inequality in Condition \ref{eq:feasible} of Theorem~\ref{thm:maximum} is equivalent to
\begin{equation}
 G(t)-G(s)\le\sum_{i=1}^d(F_i(t)-F_i(s))
 \qquad \text{for all } s < t.
 \label{eq:increments}
\end{equation}
Indeed, the finite signed measure $\sum_i\mu_{F_i}-\mu_G$ is nonnegative if and only if $H:=\sum_iF_i-G$ is nondecreasing. This follows from \(\left(\sum_i\mu_{F_i}-\mu_G\right)((s,t])
=H(t)-H(s)\) and the uniqueness of the associated Lebesgue-Stieltjes measure.

The following existence of copulas with prescribed values along a track is well known and has been established by
different methods; see \citet{CT2001} and
\citet[Section~13.3]{Sempi2017}.
Here, we provide a particularly simple and direct
probabilistic derivation from
Theorem~\ref{thm:maximum}.

\begin{corollary}[Copulas with prescribed values along a track]
\label{cor:extension}
Let $F_1,\ldots,F_d,G:\I\to\I$ be continuous and nondecreasing,
with value $0$ at $0$ and value $1$ at $1$. Suppose that
\begin{align}
 G(t)&\le F_i(t)
 && \text{for all } t\in\I,\ i\in \{1,\ldots,d\},\label{eq:track-upper}\\
 G(t)-G(s)&\le\sum_{i=1}^d(F_i(t)-F_i(s))
 &&\text{for all } 0\le s\le t\le1 .\label{eq:track-increments}
\end{align}
Then there is a $d$-copula $C$ such that
\begin{align}
 C(F_1(t),\ldots,F_d(t))=G(t)\qquad  \text{for all } t\in\I.
 \label{eq:track-extension}
\end{align}
\end{corollary}

\begin{proof}
Extend the functions \(F_i\) and \(G\) by $0$ below $0$ and $1$ above $1$.
Theorem~\ref{thm:maximum} gives $X_i\sim F_i$ with \(\max_i X_i\sim G\).
Then $U_i=F_i(X_i)$ is uniform by continuity of \(F_i\). For any $t$, we have
$\{X_i\le t\}\subseteq\{U_i\le F_i(t)\}$ and both events have
probability $F_i(t)$, so they agree almost surely. The copula $C$
of $U$ therefore satisfies
$C(F_1(t),\ldots,F_d(t))=\Pp(\max_iX_i\le t)=G(t)$. 
\end{proof}

\section{A short proof for the characterization of quasi-copulas}
\label{sec:quasi}

The following theorem is the classical result
of \citet{GQRS1999} and \citet{CT2001}.
It gives a simple analytic version of the original definition \ref{thm:quasi1} of quasi-copulas.
Below, we address the call in \citet[Section~13.3]{Sempi2017} for a \emph{simpler and more transparent} proof where we use
Corollary~\ref{cor:extension} of Theorem \ref{thm:maximum} to establish the nontrivial direction
'$\ref{thm:quasi2}\Rightarrow\ref{thm:quasi1}$'.

\begin{theorem}[Classical characterization of quasi-copulas]
\label{thm:quasi}
For a function $Q:\I^d\to\I$, the following are equivalent.
\begin{enumerate}[label = (\roman*)]
\item \label{thm:quasi1} For every track $T$, there is a $d$-copula $C_T$
 such that $C_T(u)=Q(u)$ for every $u\in T$.
\item \label{thm:quasi2} The function $Q$ has uniform marginals, is nondecreasing in each coordinate, and satisfies
\[
 |Q(u)-Q(v)|\le \sum_{i=1}^d |u_i-v_i|,\qquad u = (u_1,\ldots,u_d),v = (v_1,\ldots,v_d)\in\I^d.
\]
\end{enumerate}
\end{theorem}

\begin{proof}
For the implication \emph{$\ref{thm:quasi2}\Rightarrow\ref{thm:quasi1}$}, we now provide a substantially shorter and simpler proof than the classical approaches
in the literature; see \cite[Section~13.3]{Sempi2017}
for an overview of the original proof strategy.
To this end, fix an arbitrary track
$T \colon t\mapsto F(t)=(F_1(t),\ldots,F_d(t))$ and define
$G(t):=Q(F(t))$.
The boundary conditions and coordinatewise monotonicity
of $Q$ imply \(0\le Q(u)\le\min_{1\le i\le d}u_i\),
and hence 
\begin{align}
    G(t)\le F_i(t) \quad \text{for all } i\in \{1,\ldots,d\}, t\in [0,1].
\end{align}
Moreover, $G$ is continuous and nondecreasing,
with $G(0)=0$ and $G(1)=1$.
For $0\le s\le t\le1$, the Lipschitz condition yields
\[
0\le G(t)-G(s)
\le\sum_{i=1}^d(F_i(t)-F_i(s)).
\]
Hence, by
Corollary~\ref{cor:extension}, there exists a $d$-copula $C_T$
such that
\begin{align*}
    C_T(F(t))=G(t)=Q(F(t)) \quad \text{for all } t\in[0,1]
\end{align*}
which proves \ref{thm:quasi1}.

\smallskip
For completeness, we also prove this direction \emph{$\ref{thm:quasi1}\Rightarrow\ref{thm:quasi2}$} using standard arguments:
Therefore, consider for every $u\in[0,1]^d$ the track \(F_u(s):= 2su\), if \(0\leq s \leq 1/2\), and \(F_u(s):= u + (2s-1)(\mathbf{1}-u)\), if \(1/2 < s \leq 1\).
Then there exists a copula \(C_u\) such that \(C_u(u) = Q(u)\). In particular, for \(u = (u_k)\) with \(u_j=1\) for \(j\ne i\), we have \(Q(u) = C_u(u) = u_i\) since \(C_u\) has uniform marginals,. Hence, also \(Q\) has uniform marginals.

Now let $u,v\in[0,1]^d$ with $u\le v$
coordinatewise. Then, similar as above, we choose a polygonal track now passing
through $\mathbf{0},u,v,\mathbf{1}$ and a copula
$C_{u,v}$ agreeing with $Q$ on this track.
For a random vector $U=(U_1,\ldots,U_d)$ with distribution function $C_{u,v}$, we have
\[
\{U\le v\}\setminus\{U\le u\}
\subseteq
\bigcup_{i=1}^d\{u_i<U_i\le v_i\}.
\]
Since $U$ has uniform marginals, it follows that
\begin{align}\label{eq_mono}
0\le Q(v)-Q(u)
&=C_{u,v}(v)-C_{u,v}(u)
\le\sum_{i=1}^d
   \PP(u_i<U_i\le v_i)
=\sum_{i=1}^d(v_i-u_i).
\end{align}
Consequently, $Q$ is coordinatewise nondecreasing.

Finally, for arbitrary $u,v\in[0,1]^d$, let \(w:= u\wedge v\) and \(z:= u\vee v\). Since \(w\leq u,v\leq z\), \eqref{eq_mono} gives \(Q(w) \leq Q(u)\) and \(Q(v) \leq Q(z)\). Hence, 
\begin{align*}
    \lvert Q(u) - Q(v)\rvert \leq Q(z) - Q(w) \leq \sum_{i=1}^d (z_i-w_i) = \sum_{i=1}^d |u_i-v_i| .
\end{align*}
Thus, $Q$ satisfies the required Lipschitz
condition, proving \ref{thm:quasi2}. 
\end{proof}

\begin{remark}
Theorem~\ref{thm:maximum} also provides a simple solution
to the finite interpolation problem underlying the original
characterization of quasi-copulas in \citet{CT2001}; see also
\citet[Section~13.3]{Sempi2017}. 
Let $\mathbf{0}=u^{(0)}\le\cdots\le u^{(m)}=\mathbf{1}$
coordinatewise, and let $q_0,\ldots,q_m\in\R$ with $q_0=0$ and $q_m=1$.
Then there exists a $d$-copula $C$ satisfying \(C(u^{(k)})=q_k\) for \(k\in \{0,\ldots,m\}\)
if and only if
\[
q_k\le\min_i u_i^{(k)},\qquad
0\le q_{\ell+1}-q_\ell
\le\sum_{i=1}^d
\bigl(u_i^{(\ell+1)}-u_i^{(\ell)}\bigr)
\]
for all \(k\in \{0,\ldots,m\}\) and \(\ell\in \{0,\ldots,m-1\}\). 
Indeed, necessity follows from the boundary,
monotonicity, and Lipschitz properties of copulas.
For sufficiency, piecewise linear interpolation
of the prescribed points and values yields
functions $F_1,\ldots,F_d,G$ satisfying the
assumptions of Corollary~\ref{cor:extension},
which provides the desired copula.
\end{remark}

\section{Proof of Theorem~\ref{thm:maximum}}
\label{sec:proofs}

We make use of the concept of distributional transform as follows.
For a distribution function $F$ on $\R$, denote its
generalized inverse by
\[
F^{-1}(u):=\inf\{x\in\R:F(x)\ge u\},
\qquad u\in(0,1).
\]
Moreover, write $F(y-):=\lim_{x\uparrow y}F(x)$
for the left-side limit of $F$ at $y$.

\begin{lemma}[Randomized quantile coupling]
\label{lem:quantile}
Let $A$ and $B$ be distribution functions on $\R$.
Let $Y\sim B$ and let $V$ be uniformly distributed on
$(0,1)$, independently of $Y$. Then the randomized
distributional transform
\begin{equation}
W:=B(Y-)+V\bigl(B(Y)-B(Y-)\bigr)
\label{eq:randomized-pit}
\end{equation}
is uniformly distributed on $(0,1)$ and satisfies
\[
B^{-1}(W)=Y
\qquad\text{almost surely}.
\]
\end{lemma}

\begin{proof}
For completeness, we give the short proof of the construction in \cite{Rueschendorf2009}: Let $U$ be uniform on $(0,1)$
and put $Z=B^{-1}(U)$.
Then $Z\sim B$ and
\[
B(Z-)\le U\le B(Z)
\qquad\text{almost surely}.
\]
At every atom $z$ of $B$, the conditional distribution
of $U$ given $Z=z$ is uniform on
$(B(z-),B(z))$. On the nonatomic part, we have
$U=B(Z)$ almost surely. Consequently,
$(Z,U)$ and $(Y,W)$ have the same distribution.
This proves that $W$ is uniform and
$B^{-1}(W)=Y$ almost surely.
\end{proof}

\begin{proof}[Proof of Theorem~\ref{thm:maximum}.]
Write $\mu_i=\mu_{F_i}$ and $\gamma=\mu_G$. 
\emph{Necessity:} Let \(X_i\sim F_i\) and $M:=\max_iX_i\sim G$. Then $X_i\le M$ and thus $G(t)\le F_i(t)$.
Moreover, for every Borel set $B$,
\[
 \{M\in B\}\subseteq\bigcup_{i=1}^d\{X_i\in B\},
\]
because the maximum is attained by at least one coordinate. Hence
$\gamma\le\sum_i\mu_i$, proving \ref{eq:feasible}.

\emph{Sufficiency:} 
Suppose \ref{eq:feasible} holds. In the first step, we decompose the prescribed maximum law \(\gamma=\mu_G\) into \(d\) subprobability measures, each assigned to one of the coordinates. In the second step, we construct random variables \(X_1,\ldots,X_d\) with the prescribed marginal distributions \(X_i\sim F_i\) such that \(\max_i X_i\sim G\).

\emph{Step 1:} Let $\lambda:=\sum_i\mu_i$. Since \(\gamma\leq \lambda\), we have \(\gamma\ll \lambda\). By the Radon--Nikodym theorem, there exists a measurable \(h\geq 0\) such that $h=d\gamma/d\lambda$. In particular, \(\gamma\leq \lambda\) implies $0\le h\le1$ $\lambda$-almost surely.
Set
\begin{align}
 \nu_i(B):=\int_B h\,d\mu_i,
 \qquad p_i:=\nu_i(\R).
 \label{eq:allocation}
\end{align}
Clearly, $\nu_i\le\mu_i$. Further, $\sum_i\nu_i(B) = \int_B h \,d\left(\sum_i\mu_i\right) = \int_B h \, d\lambda = \gamma(B)$ and thus \(\sum_i\nu_i = \gamma\).
Moreover, since \(\gamma\) is a probability measure, we have \(p_i\geq 0\) and \(\sum_{i=1}^d p_i = 1\).
Now, we construct a probability measure \(\PP\) on \(\Omega = \R\times \{1,\ldots,d\}\) by
\begin{align}\label{def:p_i}
    \PP(B\times \{i\}) := \nu_i(B)
\end{align}
for every Borel set \(B\subset \R\). Indeed, \(\PP\) is a probability measure as \(\sum_i \nu_i(\R) = \gamma(\R) = 1\). Now, define the random variables \(M(x,i):= x\) and \(J(x,i):= i\).
Then, 
\begin{align}\label{eq_nu_i_prop}
    \PP(M\in B, J=i) = \nu_i(B)
\end{align}
for every Borel set $B$. In particular, $M\sim G$ and
$\Pp(J=i)=p_i$.

\emph{Step 2:} We now construct random variables \(X_1,\ldots,X_d\) with \(X_i\sim F_i\) such that \(X_J=M\) and \(X_i\le M\) for all \(i\). This ensures that \(M=\max_i X_i\) almost surely.

For every index \(i\) with $p_i<1$, we define the (conditional) probabilities
\begin{align}
    \beta_i(B)&:= \frac{\gamma(B) - \nu_i(B)}{1-p_i} = \frac{\PP(M\in B, J\ne i)}{1-p_i} = \PP(M\in B\mid J\ne i) ,\\
    \eta_i(B) &:=\frac{\mu_i(B)-\nu_i(B)}{1-p_i}.
\end{align}
Indeed, \(\beta_i\) and \(\eta_i\) are probability measures since, on the one hand, \(\gamma \geq \nu_i\) and \(\mu_i\geq \nu_i\). On the other hand, \(\beta_i(\R) = 1 = \eta_i(\R)\) since \(\gamma(\R) - \nu_i(\R) = 1-p_i\) and \(\mu_i(\R) - \nu_i(\R) = 1-p_i\) using \eqref{eq:allocation}.

Denote by \(A_i\) and \(B_i\) the distribution functions of \(\eta_i\) and \(\beta_i\), respectively.
Writing $H_i(t):=\nu_i((-\infty,t])$, we obtain
\begin{align}
A_i(t)-B_i(t)
&=\frac{F_i(t)-H_i(t)}{1-p_i}
 -\frac{G(t)-H_i(t)}{1-p_i}
=\frac{F_i(t)-G(t)}{1-p_i}\ge0,
\label{eq:cancellation}
\end{align}
where the inequality follows from \ref{eq:feasible}.
Thus, $A_i\ge B_i$ pointwise, which implies
\begin{align}\label{eq:A_iB_i_inverse}
    A_i^{-1}(u)\le B_i^{-1}(u),
\qquad u\in(0,1).
\end{align}

We now construct the random variables $X_1,\ldots,X_d$
by quantile coupling. To this end, let $V$
be a random variable uniformly distributed
on $(0,1)$, independent of $(M,J)$. Such a random variable exists after enlarging the underlying probability space to \(\Omega\times [0,1]\) equipped with the product measure \(\PP\otimes \lambda_{[0,1]}\), where \(\lambda_{[0,1]}\) is the Lebesgue measure on \([0,1]\).
Then define
\[
X_i:=
\begin{cases}
M, & \text{if } J=i,\\[1mm]
A_i^{-1}(W_i), & \text{if } J\ne i,
\end{cases}
\]
where, on the event $\{J\ne i\}$, \(W_i:=
B_i(M-)+V\bigl(B_i(M)-B_i(M-)\bigr)\).

Since $M$ has conditional distribution function $B_i$
on $\{J\ne i\}$, Lemma~\ref{lem:quantile} implies that
$W_i$ is uniformly distributed on $(0,1)$
on $\{J\ne i\}$, and thus
\[
B_i^{-1}(W_i)=M
\qquad\text{almost surely on }\{J\ne i\}.
\]
Moreover, by \eqref{eq:A_iB_i_inverse}, we have
\[
X_i=A_i^{-1}(W_i)
\le B_i^{-1}(W_i)=M
\qquad\text{almost surely on }\{J\ne i\}.
\]
Since $X_i=M$ on $\{J=i\}$, we obtain
$X_i\le M$ almost surely.

It remains to verify the prescribed marginal distributions.
By Lemma~\ref{lem:quantile}, the conditional distribution
of $X_i$ given $J\ne i$ is $\eta_i$, with distribution
function $A_i$. Therefore, for every $t\in\R$,
\begin{align*}
\PP(X_i\le t)
&=\PP(X_i\le t,J=i)
  +\PP(X_i\le t,J\ne i)\\
&=\PP(M\le t,J=i)
  +(1-p_i)\, \PP(X_i\leq t\mid J\ne i)\\
&=H_i(t)+(1-p_i)A_i(t)\\
&=H_i(t)+F_i(t)-H_i(t)\\
&=F_i(t),
\end{align*}
where we use \eqref{eq_nu_i_prop} for the third equality.
Thus, $X_i\sim F_i$ whenever $p_i<1$.

Finally, if there is an index $i$ with $p_i=1$, then \(\mu_i = \nu_i = \gamma\) and, by \eqref{eq_nu_i_prop}, \(J = i\) almost surely. Hence, setting $X_i=M$
also yields $X_i\sim F_i$ in this case.

We have therefore constructed random variables
$X_1,\ldots,X_d$ with \(X_i\sim F_i\) such that $X_i\le M$ for every $i$
and $X_J=M$ almost surely. It follows that \(\max_{1\le i\le d}X_i=M\sim G\),
which completes the proof. 
\end{proof}

\section*{Acknowledgements}

During the preparation of this manuscript, the author used ChatGPT (OpenAI) for assistance with mathematical reasoning, including the development and verification of the constructive coupling argument, as well as literature searches, exposition, and language editing. All AI-assisted suggestions were critically assessed and revised by the author, who takes full responsibility for the mathematical correctness, originality, and content of the manuscript.
This research was funded in whole or in part by the
Austrian Science Fund (FWF) [10.55776/PAT1669224], project
\emph{Stochastic orders for functional dependence}.

\bibliographystyle{plainnat}
\bibliography{Quasi-Copulas}

\end{document}